\documentclass[J,12pt]{amsart}
\usepackage{epsfig,epsf,amsmath,amssymb}
\newtheorem{theorem}{Theorem}[section]
\newtheorem{corollary}[theorem]{Corollary}
\newtheorem{definition}{Definition}
\newtheorem{example}[theorem]{Example}
\newtheorem{proposition}[theorem]{Proposition}
\newtheorem{remark}[theorem]{Remark}
\newtheorem{lemma}[theorem]{Lemma}

\begin{document}

	\author[Sang-Eon Han]{Sang-Eon~Han}
	
	\title[]{Remarks on the digital-topological $k$-group structures and the development of the $AP_1$-$k$- and $AP_1^\ast$-$k$-group}

	\maketitle
	
	\begin{center}
		{ Department of Mathematics Education,
			Jeonbuk National University, Republic of Korea.\\
			E-mail address: sehan@jbnu.ac.kr, sehan202058@gmail.com} 
		
	\end{center}

	\begin{abstract} In the literature of a digital-topological ($DT$-, for brevity) group structure on a digital image $(X,k)$, roughly saying, two kinds of methods are shown. Given a digital image $(X,k)$, the first one, named by a $DT$-$k$-group, was established in 2022 \cite{H10} by using both the $G_{k^\ast}$- or $C_{k^\ast}$-adjacency \cite{H10} for the Cartesian product $X^2:=X \times X$  and the $(G_{k^\ast},k)$- or $(C_{k^\ast},k)$-continuity for the multiplication $\alpha:X^2 \to X$ \cite{H10}.
		 The second one with the name of $NP_i$-$DT$-groups, $i \in \{1,2\}$, was 
		discussed in 2023 \cite{LS1} by using the $NP_i(k,k)$-adjacency for $X^2$ in \cite{B1} and the $(NP_i(k,k), k)$-continuities of the multiplication $\alpha:X^2 \to X$, $i\in \{1,2\}$. However, due to some ambiguity of the $NP_u(k_1,k_2, \cdots, k_v)$-adjacency in \cite{B1,B2}, the $AP_u(k_1,k_2, \cdots, k_v)$-adjacency was recently developed as an improvement of the $NP_u(k_1,k_2, \cdots, k_v)$-adjacency (see Section 4). Besides, we also develop an $AP_u^\ast(k_1,k_2, \cdots, k_v)$-adjacency. For a digital image $(X, k)$, in case an $AP_1(k,k)$-adjacency on $X^2$ exists, we formulate both  an $AP_1$-$k$- and an $AP_1^\ast$-$k$-group. Finally, we show that an $AP_1^\ast$-$k$-group is equivalent to a Han's $DT$-$k$-group in \cite{H10}. 
	\end{abstract}
	
		\thanks{2010 Mathematics Subject Classification: 22A05,22A10,54C08, 54H11,68U10\\
		Keywords: Digital-topological ($DT$-, for brevity) $k$-group, digital $k$-adjacency, $C$-compatible $k$-adjacency, $G_{k^\ast}$-adjacency, $AP_u(k_1, \cdots, k_v)$-adjacency, $AP_u^\ast(k_1, \cdots, k_v)$-adjacency, $AP_1$-$k$-group, $AP_1^\ast$-$k$-group.}

	
	\maketitle
		\newpage
	
	\section{\bf Introduction}\label{s1}
	For two integers $c,d \in {\mathbb Z}$: the set of integers, we will follow the term 
	$[c,d]_{\mathbb Z}=\{t \in {\mathbb Z}\,\vert\, c\leq t \leq d\}$. 
	Besides, for a set $X$, we will be using the notation $\sharp X$ to indicate the cardinal number of $X$ and the notation $\lq\lq$$:=$" (resp. $\Leftrightarrow$) will be used for introducing a new term (resp. if and only if). Besides, much of the material that appears in the present paper is often quoted or paraphrased from \cite{B1,B2,H1,H2,H6,H7,H10,H11,H12,LS1}.\\
	
	Recently, it turns out that each of the $n$-dimensional Khalimsky topological space \cite{KKM1} and the Marcus-Wyse topological space \cite{MW1} is not a topological group \cite{H11}.
	Hence the study of the existence of a digital-topological ($DT$-, for simplicity) group structure on a digital image $(X,k)$ is meaningful from the viewpoint of digital topology.
	According to the literature of a $DT$-group structure on a digital image $(X,k)$, two kinds of approaches are shown in \cite{H10,LS1}. Precisely, given a digital image $(X, k)$, in order to formulate a $DT$-group structure on $(X,k)$, we need some special kinds of adjacencies or relations on the product $X^2:=X\times X$ and further, some continuities for the multiplication $\alpha$ of (1.1) below (in detail, see Sections 5 and 6 in the present paper)  
	$$\alpha: X^2 \to X. \eqno(1.1)$$
   Then, to support some special kinds of continuities of the multiplication $\alpha$ of (1.1), $(G_{k^\ast}, k)$-continuity \cite{H10} and $(NP_u(k,k),k)$-continuity \cite{LS1} were considered (see Sections 5 and 6). After comparing the two approaches (or ideas) with each other, we observe that each of them has its own feature. However, the recent paper \cite{H12} pointed out some lacks of clarity on the $NP_u(k_1,k_2, \cdots, k_v)$-adjacency discussed in \cite{B1,B2}, $u \in [1,v]_{\mathbb Z}$ (see Section 4).
       Hence the paper \cite{H12} developed an $AP_u(k_1,k_2, \cdots, k_v)$-adjacency as an improvement of the $NP_u(k_1,k_2, \cdots, k_v)$-adjacency in \cite{B1,B2}, $u \in [1,v]_{\mathbb Z}$ (see Sections 5 and 6).
   Based on both an $AP_u(k,k)$- and an $AP_u^\ast(k,k)$-adjacency for the product $X^2$, the paper develops both an $AP_1$-$k$-group and an $AP_1^\ast$-$k$-group on a digital image $(X,k)$ (see Section 6).\\
   Finally, we obtain the following:\\
   $\bullet$ An equivalence between an $AP_1^\ast$-$k$-group  and a Han's $DT$-$k$-group in \cite{H10} using both the $C_{k^\ast}$-adjacency on the product $X^2$ and the $(C_{k^\ast}, k)$-continuity for the multiplication $\alpha_1^\prime:(X^2, C_{k^\ast}) \to (X,k)$  (see Section 6). \\   
    $\bullet$ Development of an $AP_1$-$k$-group and an $AP_1^\ast$-$k$-group  and their comparison (see Sections 5 and 6). \\    
    $\bullet$ The $DT$-$k$-group \cite{H10} based on both the $G_{k^\ast}$-adjacency on the product $X^2$ and the $(G_{k^\ast}, k)$-continuity for the multiplication for a $DT$-$k$-group is proved to be the most generalized $DT$-$k$-group in the literature of the study of a $DT$-group structure on a digital image $(X,k)$ (see Sections 5 and 6). 
    Then we strongly note that the pair  $(X^2, G_{k^\ast})$ is a relation set \cite{H10} that need not be a digital image with the typical adjacencies in the digital setting. However, we note that each of the pairs  $(X^2, AP_1(k,k))$  and $(X^2, AP_1^\ast(k,k))$ (see Section 6) is a digital image because each of $AP_1(k,k)$ and $AP_1^\ast(k,k)$ is one of the typical adjacencies of (2.1).\\
    
The paper is organized as follows:
Section 2 refers to some terminology for further studies in the paper. 
Section 3 mentions some utilities of a $C_{k^\ast}$- and a $G_{k^\ast}$-adjacency for formulating a $DT$-$k$-group. Besides, it refers to some reason why we take the number (subscript) $k^\ast$ of both $C_{k^\ast}$- and $G_{k^\ast}$-adjacency to establish a $DT$-$k$-group on a digital image $(X,k)$.
Section 4 discusses about some unclearness of the $NP_u(k_1,k_2, \cdots, k_v)$-adjacency in \cite{B1,B2} and referes to the $AP_u(k_1,k_2, \cdots, k_v)$-adjacency as an improvement of the $NP_u(k_1, \cdots, k_v)$-adjacency in \cite{B1,B2}.
Section 5 compares the two $DT$-group structures in \cite{H10} and \cite{LS1} and suggests some ideas for improving the $DT$-group structure in \cite{LS1}.	
Section 6 proposes both an $AP_1$-$k$- and an $AP_1^\ast$-$k$-group structure. In particular,  we show that the  $AP_2(k,k)$-adjacency on the product $X^2$ is very restrictive to establish a $DT$-group structure on $(X,k)$. For instance, even $({\mathbb Z}, 2)$ cannot be an $AP_2$-$2$-group (see the proof of Remark 6.11) because the $AP_2(2,2)$-adjacency on ${\mathbb Z}^2$ is the normal $8$-adjacency on ${\mathbb Z}^2$.  Besides, an $AP_1$-$k$-group (resp. $AP_1^\ast$-$k$-group) will be proved to be a generalization of (resp. be equivalent to) a Han's $DT$-$k$-group using a $C_{k^\ast}$-adjacency on $X^2$ (compare the map $\alpha_1^\prime$ of (5.1) and the map $\alpha_3$ of (6.2)) (see Section 6). In addition, comparison among several types of $DT$-group structures mentioned earlier will be done.
Finally, we investigate some properties of an $AP_1$-$k$-group.

\section{\bf Some terminology }\label{s2}

In order to effectively deal with a set $X \subset {\mathbb Z}^n$ from the viewpoint of digital topology, 
$k$-adjacencies of  ${\mathbb Z}^n, n \in {\mathbb N}$: the set of positive integers, have been formulated, as stated in (2.1) below.
The typical $k$-adjacencies of ${\mathbb Z}^n, n \in [1,3]_{\mathbb Z}$, as low dimensional cases, were initially developed by Rosenfeld \cite{R1,R2,R3}.
As a generalization of them, for ${\mathbb Z}^n, n \in {\mathbb N}$, the papers \cite{H1,H2} initially (in detail see \cite{H8}) established the following: for $t \in [1,n]_{\mathbb Z}$,
the distinct points $p = (p_i)_{i \in [1,n]_{\mathbb Z}}$ and $q = (q_i)_{i \in [1,n]_{\mathbb Z}}$  in  ${\mathbb Z}^n$
are said to be $k(t, n)$-adjacent if at most $t$ of their coordinates  differ by $\pm1$ and the others coincide. Then, in the present paper, we may write $(p,q) \in k(t,n)$ because the $k(t,n)$-adjacency is a kind of relation on ${\mathbb Z}^n$. i.e., 
the pair $(X, k(t,n))$ can be assumed to be a relation set, as usual. Hereinafter, we will often write $k:=k(t,n)$ if there is no danger of confusion as in (2.1).

According to the above mentioned criterion for establishing the $k(t, n)$-adjacency, the $k(t, n)$-adjacencies on ${\mathbb Z}^n, n \in {\mathbb N}$, were formulated \cite{H2,H8} as follows:
$$k:=k(t, n)=\sum_{i=1}^{t} 2^{i}C_{i}^n, \text{where}\,\, C_i ^n= {n!\over (n-i)!\ i!}. \eqno(2.1)$$
For instance \cite{H1,H2,H9},
$$(n, t, k) \in \left \{
\aligned
& (1,1,2),\\
& (2,1,4), (2,2,8),\\
& (3,1,6), (3,2,18), (3,3,26),  \\
&  (4, 1, 8), (4, 2, 32), (4, 3, 64), (4, 4, 80),   \\
&  (5, 1, 10),  (5, 2, 50), (5, 3, 130), (5, 4, 210), (5, 5, 242),\, \,\text{and}  \\
&  (6, 1, 12),  (6, 2, 72),  (6, 3, 232), (6, 4, 472), (6, 5, 664), (6, 6, 728). 
\endaligned
\right\} \eqno(2.2)
$$
For a set $X \subset {\mathbb Z}^n, n \in {\mathbb N}$, with one of the $k$-adjacencies of (2.1), hereinafter, we call $(X, k)$ a digital image (or digital object), as usual.
We clearly note that the $c_u$-adjacencies in \cite{B1} for $X \subset {\mathbb Z}^n, n \in {\mathbb N}$, and $u \in [1,u]_{\mathbb Z}$, are equal to those of the digital adjacencies of (2.1), i.e., $\lq\lq$$c_u$-adjacency for $X \subset {\mathbb Z}^n$" is equal to the $\lq\lq$$k:=k(t,n)$-adjacency of (2.1), where $t=u$".

Given $(X, k)$, it is said to be  $k$-connected (or $k$-path connected)
if, for any two points $x, y \in X$, there exists a finite sequence (or a $k$-path) $(x_i)_{i \in [0, n]_{\mathbb Z}} \subset X$ such that $x_0=x$, $x_n=y$, and $x_i$ is  $k$-adjacent to $x_{i+1}$ or  $x_i=x_{i+1}$, $i \in [0,n-1]_{\mathbb Z}$ \cite{KR1}.
In this regard, a digital image  $(X, k)$ with a singleton is obviously $k$-connected for any $k$-adjacency. 
Besides, a simple closed $k$-curve (or simple $k$-cycle, for brevity) with $l$ elements on ${\mathbb Z}^n$, denoted by $SC_k^{n,l}, 4 \leq l \in {\mathbb N}$ \cite{H1,H2,H3,KR1},
was initially defined as a sequence $(s_i)_{i \in [0, l-1]_{\mathbb Z}}$ on ${\mathbb Z}^n$, where
$s_i$ and $s_j$ are $k$-adjacent if and only if $\vert i-j\vert=\pm1(mod\, l)$ (or $i=(j\pm1) (mod\, l))$ \cite{H2,KR1}.  
Indeed, owing to the notion of $SC_k^{n,l}$ above, it is clear that a circular ordering \cite{Ko1} on $SC_k^{n,l}$ is assumed because $SC_k^{n,l}$ is a sequence as stated earlier, i.e., for any $i \in [0,l-1]_{\mathbb Z}, x_i \preccurlyeq x_{i+1} (mod\, l)$. 
To be specific, consider the order preserving map 
$f: ({\mathbb N}\cup \{0\}, \leq) \to (SC_k^{n,l}=(x_i)_{i \in [0, l-1]_{\mathbb Z}}, \preccurlyeq)$ defined by $f(i)=x_{i (mod\, l)}$. Then it is obvious that  $(SC_k^{n,l}, \preccurlyeq)$ is a cyclic ordered set.\\

\begin{remark} \cite{H10} $(SC_k^{n,l}, \ast)$ is an abelian group, where for two elements $x_i$ and $x_j$ in $SC_k^{n,l}$, the operation
	$x_i \ast x_j$ is defined as $x_{i+j(mod\,l)}$.	
	\end{remark}

A digital $k$-neighborhood of the point $x_0$ with radius $1$ \cite{H1,H2} was defined as follows:
$$N_k(x_0, 1)=\{x \in X \,\vert\,\,x\,\,\text{is}\,\,k\text{-adjacent to}\,\,x_0\} \cup \{x_0\} \subset (X,k). $$

However, for convenience, in the present paper, given a digital image $(X,k)$, we will be taking the notations
$$N_k(x_0):=N_k(x_0, 1)\,\,\text{and}\,\, \\
N_k^\ast(x_0):=N_k(x_0) \setminus \{x_0\}. \eqno(2.3) $$

These $k$-neighborhoods of (2.3) have effectively been used to formulate many types of adjacencies in digital topology (see Propositions 2.3 and 2.5, (2.8), (3.2), (4.3), and so on).
Also, the digital continuity established by Rosenfeld \cite{R3} can also be represented in terms of the digital $k$-neighborhood of (2.3) (see (2.4) below). More precisely, let us first recall that the paper \cite{R3} initially characterized 
the $(k_0, k_1)$-continuity of a map $f:(X,k_0) \to (Y,k_1)$ as follows:
A function $f:(X,k_0) \to (Y,k_1)$ is $(k_0, k_1)$-continuous if and only if, for any two points $x$ and $x^\prime$ which are $k_0$-adjacent in $(X,k_0)$, $f(x)=f(x^\prime)$ or $f(x)$ is $k_1$-adjacent to $f(x^\prime)$ (see Theorem 2.1 of \cite{R3}). Besides,
we obviously see that a map $f:(X,k_0) \to (Y,k_1)$ is $(k_0, k_1)$-continuous if and only if, for any $k_0$-connected subset $A \subset (X,k_0)$, $f(A)$ is $k_1$-connected in $(Y,k_1)$ (see Theorem 2.4 of \cite{R3}).
Indeed, it is clear that this continuity can be represented by using a digital $k$-neighborhood of (2.3), which can be very effective to study digital images (see Sections 3, 4, 5, and 6), as follows \cite{H2,H3,R3}.\\
	Assume $X \subset {\mathbb Z}^{n_1}$ and $Y \subset {\mathbb Z}^{n_2}$.
$$\left \{
\aligned & \text{A mapping}\,\,f: (X,k_1) \to (Y, k_2)\,\,\text{is}\,\,(k_1, k_2)\text{-continuous}\\
&\text{if and only if for each}\,\,x\in X, f(N_{k_1}(x))\subset  N_{k_2}(f(x)).
\endaligned
\right\}\eqno(2.4)  $$	

In (2.4), in the case that  $n_1 = n_2$ and $k_1 = k_2$, we say that it is {\it $k_1$-continuous}.\\

	Let us now consider a Berge's and Harary's adjacency for a graph product in \cite{Be1,H14} in the digital setting such as a $C$-compatible $k$-adjacency \cite{H7} and a normal $k$-adjacency \cite{H2} for a product of two digital images.
	To be specific, motivated by a Berge's adjacency for a  graph product in \cite{Be1}, a normal $k$-adjacency was developed in \cite{H2} (see Definition 1 in the present paper). Even though these two adjacencies have the same logic, we can obviously make a distinction between them (see Remark 2.2 in the present paper).\\
		
		To make the paper self-contained, let us first recall the normal $k$-adjacency for a product of digital images in \cite{H2}. \\
	
	\begin{definition} \cite{H2} For two digital images $(X, k_1)$ on
		${\mathbb Z}^{n_1}$ and $(Y, k_2)$ on ${\mathbb Z}^{n_2}$, a $k$-adjacency on the product $X\times Y \subset {\mathbb Z}^{n_1+n_2}$ is defined, according to (2.1), as follows:
		for distinct points $p:=(x, y), q:=(x^{\prime}, y^{\prime}) \in {X \times Y}$,
		$p$ is $k$-adjacent to $q$ if and only if		
		$$\left \{
		\aligned  
		&(1)\, y=y^{\prime}\,\,\text{and}\,\,x\,\,\text{is}\,\,k_1\text{-adjacent to}\,\,
		x^{\prime},\,\,\text{or}\\
		&(2)\, x=x^{\prime}\,\,\text{and}\,\,y \,\,\text{is}\,\,k_2\text{-adjacent to}\,\,
		y^{\prime},\,\,\text{or}\\				
		&	(3) \,x\,\,\text{is}\,\,k_1\text{-adjacent to}\,\,x^{\prime}\,\,\text{and}\,\,y\,\,\text{is}\,\,k_2\text{-adjacent to}\,\,y^{\prime}.
		\endaligned
		\right\}\eqno(2.5)  $$	
		
		Then we say that the product $X\times Y$ has a normal $k$-adjacency derived from $(X, k_1)$ and $(Y, k_2)$, i.e., digital image $(X\times Y,k)$.
	\end{definition}

	\begin{remark} \cite{H2,H5,H6,H13,H14} A normal $k$-adjacency is slightly different from a Berge's adjacency (or strong adjacency). The difference between them is attributed to a distinction between $\lq\lq$adjacency"in graph theory \cite{Be1} and $\lq\lq$$k$-adjacency" in digital topology \cite{H2,H7}.
		Namely, while every graph product always has a  Berge's adjacency, not every digital product always has a normal $k$-adjacency (see Theorem 2.6 in the present paper).
			\end{remark}

	\begin{proposition} \cite{H6,H7} Assume two digital images $(X, k_1)$ on
		${\mathbb Z}^{n_1}$ and $(Y, k_2)$ on ${\mathbb Z}^{n_2}$. Given a  product $X \times Y  \subset {\mathbb Z}^{n_1+n_2}$ with a $k$-adjacency,
		the following are equivalent. \\
		(1) A relation set $(X \times Y, k)$ is a  product with a normal $k$-adjacency. \\
		(2) For any point $p:=(x, y) \in X \times Y$, $$N_k(p)= N_{k_1}(x) \times N_{k_2}(y). \eqno(2.6)$$
	\end{proposition}
	
	Motivated by a Harary's adjacency \cite{H14}, a $C$-compatible $k$-adjacency for a product of two digital images was also established \cite{H6,H7} (see also \cite{H5}) as follows:
	
	\begin{definition} \cite{H6,H7} For two digital images $(X, k_1)$ on
		${\mathbb Z}^{n_1}$ and $(Y, k_2)$ on ${\mathbb Z}^{n_2}$, a $k$-adjacency on the  product $X\times Y \subset {\mathbb Z}^{n_1+n_2}$ is defined as follows:
		For distinct points $p:=(x, y), q:=(x^{\prime}, y^{\prime}) \in {X \times Y}$,
		$p$ is $k$-adjacent to $q$ if and only if
		$$\left \{
		\aligned  
		&(1)\, y=y^{\prime}\,\,\text{and}\,\,x\,\,\text{is}\,\,k_1\text{-adjacent to}\,\,
		x^{\prime},\,\,\text{or}\\
		&(2)\, x=x^{\prime}\,\,\text{and}\,\,y \,\,\text{is}\,\,k_2\text{-adjacent to}\,\,
		y^{\prime}.\\				
			\endaligned
		\right\}\eqno(2.7)  $$	
						Then, we say that the product $X\times Y$ has a $C$-compatible $k$-adjacency derived  from $(X, k_1)$ and $(Y, k_2)$, i.e., digital image $(X\times Y,k)$.
	\end{definition}

Let us now recall a distinction between a Harary's adjacency and a $C$-compatible $k$-adjacency, as follows:
	
	\begin{remark} \cite{H6,H7} A $C$-compatible $k$-adjacency of Definition 2 is different from a Harary's  adjacency in graph theory \cite{H14}. The decisive difference between them comes from the distinction between  $\lq\lq$adjacency" in graph theory \cite{H14} and $\lq\lq$$k$-adjacency" of (2.1) \cite{H6,H7}. Namely, given two graphs $X$ and $Y$, while every graph product $X\times Y$ always has a Harary's adjacency, not every digital product always has a $C$-compatible $k$-adjacency (see Theorem 2.6 in the present paper).
			\end{remark}
		
				Then a  product with a $C$-compatible $k$-adjacency can be represented as a relation set as follows:
	
	\begin{proposition} \cite{H6,H7} Assume two digital images $(X, k_1)$ on
		${\mathbb Z}^{n_1}$ and $(Y, k_2)$ on ${\mathbb Z}^{n_2}$. For the product $X \times Y  \subset {\mathbb Z}^{n_1+n_2}$ with a $k$-adjacency,
		the following are equivalent. \\
		(1) The relation set $(X \times Y, k)$ is a  product with a $C$-compatible $k$-adjacency. \\
		(2) For any point $p:=(x, y) \in X \times Y$, 
		$$N_k(p)= (N_{k_1}(x) \times \{y\}) \cup (\{x\} \times N_{k_2}(y)). \eqno(2.8)$$
	\end{proposition}
	
	In view of Definitions 1 and 2, the following is obtained.
	\begin{theorem} Not every digital image $(X,k)$ always has a $C$-compatible $k$-adjacency or  a normal $k$-adjacency for the product $X^2$. 
			\end{theorem}
		Even though the assertion was done in \cite{H6,H7}, since the fact plays an important role in studying $DT$-group theory, we now confirm the assertion by using  a good example, as follows:
		
		\begin{proof} Consider the digital image $X:=MSC_{18}$ with an $18$-adjacency \cite{H4} (see (2.9)), where
			
			$$MSC_{18}:= \left \{
			\aligned
			& x_0=(0,0,0), x_1=(1,-1,0), x_2=(1,-1,1),\\
			& x_3=(2,0,1), x_4=(1,1,1), x_5=(1,1,0)
			\endaligned
			\right\} \eqno(2.9)
			$$
			and the product $MSC_{18}\times MSC_{18} \subset {\mathbb Z}^6$ from (2.9).
			Then, we note that $MSC_{18}$ is a kind of $SC_{18}^{3,6}$ and further,  it is obvious that none of the $C$-compatible $k$-adjacency and the normal $k$-adjacency on the product exists, $k \in \{12,72,232, 472,664, 728\}$ (see (2.1), (2.2), (2.6), and (2.8)). It is clear that $MSC_{18}\times MSC_{18}$ does not satisfy both (2.6) and (2.8). 
			For instance, it is obvious that the point $(x_2,x_2) \in MSC_{18}\times MSC_{18}$ from (2.9) does not satisfy both (2.6) and (2.8). In detail, for any $k \in \{12,72,232, 472,664, 728\}$, we clearly obtain
					
			$$\left \{
			\aligned
			& N_k((x_2,x_2)) \neq (N_{18}(x_2) \times \{x_2\}) \cup  (\{x_2\} \times N_{18}(x_2)),\,\,\text{and}\\
			& N_k((x_2,x_2)) \neq N_{18}(x_2)\times N_{18}(x_2),
			\endaligned
			\right\} \eqno(2.10)
			$$
			which implies that $MSC_{18}\times MSC_{18}$ does not have both the $C$-compatible $k$-adjacency and the normal $k$-adjacency.
								\end{proof}

	\section{\bf Utilities of both the $C_{k^\ast}$- and the $G_{k^\ast}$-adjacency for establishing a $DT$-$k$-group}\label{s3}
		\subsection{\bf Why do we take the $C_{k^\ast}$- and $G_{k^\ast}$-adjacency to define a $DT$-$k$-topological group?}\label{s3-1}
		
For a digital image $(X, k)$, the notions of $G_{k^\ast}$- and $C_{k^\ast}$-adjacency on the product $X^2$ are essential to establishing a $DT$-$k$-group \cite{H10}. Indeed, a $C_{k^\ast}$-adjacency is more rigid than a $G_{k^\ast}$-adjacency (see Corollary 4.7 of \cite{H10}).
To confirm some relationships between the $C_{k^\ast}$-adjacency and the $G_{k^\ast}$-adjacency in \cite{H10} more precisely, let us first recall the $G_{k^\ast}$-adjacency on $X^2$ because it is broader than the $AP_1^\ast(k,k)$-adjacency on $X^2$ that will be discussed in Section 4. 
In particular,  the reader should be careful to note that a $G_{k^\ast}$-adjacency need not be one of the adjacencies of (2.1). Indeed, it is just a relation on a product as stated in \cite{H10}, as follows: 
 
 \begin{definition} \cite{H10} (1) Assume two digital images $(X, k_1:=k(t_1,n_1))$ and $(Y, k_2:=k(t_2,n_2))$.
  For distinct points $p:=(x, y), q:=(x^{\prime}, y^{\prime}) \in X\times Y\subset {\mathbb Z}^{n_1+n_2}$, we say that the point
 	$p$ is related to $q$ if they satisfy the statement of (2.7) (see Definition 2). 
 	Namely, this relation need not be the relation of Definition 2 using $\lq\lq$ if and only if". 
 	Then this relation is called a generalized $C$-compatible ($GC$-, for brevity) adjacency on $X\times Y$ derived from $(X, k_1)$ and $(Y, k_2)$. \\
 	(2) Under the relation set $(X\times Y, GC\text{-adjacency})$ from (1) above, just using  the $k^\ast$-adjacency of ${\mathbb Z}^{n_1+n_2}$ of (2.1), where $k^\ast:=k(t, n_1+n_2)$ and $t=max\{t_1,t_2\}$, 
 	
 	$$\left \{
 	\aligned  
 	& \, p\,\,\text{is called}\,\,G_{k^\ast}\text{-adjacent to}\,\,q\,\,\\
 	&\text{derived from}\,\, (X, k_1)\,\, \text{and}\,\,(Y, k_2),		
 	\endaligned
 	\right\}\eqno(3.1)  $$
 	and we use the notation $(X\times Y, G_{k^\ast})$ as a relation set. Namely, the $G_{k^\ast}$-adjacency is just another presentation of the $GC$-adjacency by using the $k^\ast$-adjacency.
 	 \end{definition}
  As a special case of Definition 3, we may assume $(X, k_1)=(Y, k_2)$.
 Then we can also deal with this case in the same way as in Definition 3.

  	Afte comparing the adjacencies of Definitions 2 and 3, it is clear that the $GC$-adjacency is broader than the adjacency of Definition 2 because the former uses only the part  $\lq\lq$ if" as in Definition 2.
  More precisely, accoring to Definitions 2 and 3, we can confirm some difference between a $G_{k^\ast}$-adjacency and a $C$-compatible $k$-adjacency, as follows:\\  
  $(\star 1)$ Unlike the $C$-compatible $k$-adjacency of Definition 2, a $G_{k^\ast}$-adjacency always and uniquely exists as a relation on $X\times Y$ derived from $(X, k_1)$ and $(Y, k_2)$.\\
  	 $(\star 2)$ The reader should be careful to note that why the number(subscript) $k^\ast$ of the $G_{k^\ast}$-adjacency is taken (see (3.1)), i.e., $\exists$ $(X\times Y, G_{k^\ast})$ as a relation set that need not be a digital image. The reason is the following: Even though a $GC$-adjacency
  	 need not be a digital adjacency of (2.1), in view of Definition 3(1), we can see that the  $GC$-adjacency or $G_{k^\ast}$-adjacency of $X\times Y$  is the most similar (or closest) to the  $k^\ast$-adjacency of ${\mathbb Z}^{n_1+n_2}$ of (2.1).
  	  Hence the paper \cite{H10} used the number $k^\ast$ for the $G_{k^\ast}$-adjacency. \\
  	  $(\star 3)$ It is obvious that a $G_{k^\ast}$-adjacency need not be equal to the $k^\ast$-adjacency (see Example 3.1 below). Indeed, the number $k^\ast$ of the $G_{k^\ast}$-adjacency  is just an auxiliary quantity to indicate some closeness between the $GC$-adjacency and the $k^\ast$-adjacency of (2.1).
   	To be specific, based on Definition 3, for a point $p \in X\times Y$,
   	let   	   		 	
   		$$
   	\begin{cases}
   	& C:=\{q \in \vert\, q\,\,\text{is}\,\,G_{k^\ast}\text{-adjacent to}\,\, p\}\,\,\text{and}\,\, \\
   	& D:=\{q \in \vert\, q\,\,\text{is}\,\,k^\ast\text{-adjacent to}\,\, p\}.
   	\end{cases} $$	
   	   	
   	Then it is clear that 
   	$\sharp C \leq \sharp D$ owing to $k^\ast:=k(t, n_1+n_2)$ as in Definition 3(2). 
   	This confirms that a $G_{k^\ast}$-adjacency need not be equal to a $k^\ast$-adjacency on $X\times Y$ \cite{H10}. Indeed, it is just a relation on $X\times Y$ as stated in Definition 3. \\
   	$(\star 4)$ Without the $G_{k^\ast}$-adjacency, we have some difficulties in establishing a $DT$-group structure on a digital image $(X, k)$ (see Section 5 and Definition 10 compared to Definition 12) because even some $SC_k^{n,l}$ might not be a $DT$-$k$-group because it need not have a $C$-compatible $k^\prime$-, a normal $k^\prime$-adjacency, or an $AP_u(k,k)$-adjacency, $u \in \{1,2\}$ (see Theorem 2.6 and Section 5 and Remark 6.5). 
   	
 	\begin{example} While $SC_4^{2,8} \times SC_8^{2,6} \subset {\mathbb Z}^4$ does not have any  $C$-compatible $k$-adjacency, it has the $G_{k^\ast}$-adjacency, i.e., $\exists$ $(SC_4^{2,8} \times SC_8^{2,6}, G_{k^\ast})$ as a relation set, where $k^\ast:=k(2,4)=32$. In this case  we note that the relation $G_{k^\ast}$ is not equal to the $32$-adjacency. Many examples for a  $G_{k^\ast}$-adjacency appeared in \cite{H10}.
 	\end{example}
 	
 	 	 Based on the $G_{k^\ast}$-adjacency,  the following $G_{k^\ast}$-neighborhood of a point of  $X\times Y$ is defined.
 
 \begin{definition} \cite{H10} Assume $(X\times Y, G_{k^\ast})$.
 	For an element $p \in X \times Y$, the $G_{k^\ast}$-neighborhood of $p$ is defined, as follows:
 	
 	$$\left \{
 	\aligned  
 	& 	N_{G_{k^\ast}}^\ast(p):=\{q \in X \times Y\,\vert \, q \,\,\text{is}\,\,G_{k^\ast}\text{-adjacent to}\,\,p\}\,\, \text{and}\,\\
 	& N_{G_{k^\ast}}(p):=N_{G_{k^\ast}}^\ast(p) \cup \{p\}.
 	\endaligned
 	\right\}\eqno(3.2)  $$	
 	 \end{definition}

\begin{definition} \cite{H10} Given $(X, k_1)$ on ${\mathbb Z}^{n_1}$ and 
$(Y, k_2)$ on ${\mathbb Z}^{n_2}$, $k_i:= k(t_i, n_i), i \in \{1, 2\}$, assume a $C$-compatible 
	$k$-adjacency on $X\times Y$. 
	Then, if the adjacency $k$ of Definition 2 is equal to $k^\ast:=k(t, n_1+n_2), \,\,\text{where}\,\,t=\text{max}\,\{t_1, t_2\}$, then for the points $p$ and $q$ in Definition 2, we say that $p$ is $C_{k^\ast}$-adjacent to $q$ and the digital image $(X\times Y, C_{k^\ast})$ is derived from the given two digital images $(X, k_1)$ and  $(Y, k_2)$, where $C_{k^\ast}=k^\ast$.
\end{definition}
	As a special case of Definition 5, we may assume $(X, k_1)=(Y, k_2)$.
Then we can also deal with this case in the same way as in Definition 5.\\

In view of Definitions 2 and 5, the following is obtained.
\begin{proposition} With Definition 5, we obtain
	
$$	\begin{cases}
		& C_{k^\ast}=k^\ast=\text{min}\{k\,\vert \,\text{the pair}\,(X\times Y,k)\,\, \text{is a product}\\
		&\text{with a} \,\,C\text{-compatible}\,\,k\text{-adjacency derived from}\,\,(X,k_1)\,\,\text{and}\,\,(Y, k_2)\}.
	\end{cases} $$	
	\end{proposition}

Let us now explain both the choice of the number(subscript) $k^\ast$ of a $C_{k^\ast}$-adjacency and Proposition 3.2 more precisely, as follows:
\begin{remark} (1) Let us consider 
	$SC_8^{2,4}:=(s_i)_{i \in [0,3]_{\mathbb Z}}$, where $s_0=(0,0), s_1=(1,-1), s_2=(2,0), s_3=(1,1)$.
	Owing to (2.8), we observe that the product $SC_8^{2,4} \times SC_8^{2,4} \subset {\mathbb Z}^4$ has the two kinds of $C$-compatible $k$-adjacencies such as $k \in \{k(2,4)=32, k(3,4)=64\}$.
	Then it is clear that the $C_{k^\ast}$-adjacency of $SC_8^{2,4} \times SC_8^{2,4}$ is exactly the $32$-adjacency because $k(2,4)=32$ where 
	$32=k^\ast=k(t, 4), t=\text{max}\,\{2, 2\}=2$.\\
	Namely, in case $X\times Y$ has several kinds of $C$-compatible $k$-adjacencies, we observe the inequality $k^\ast \leq k$ (see Proposition 3.2).\\	
	(2) The number $k^\ast$ of Definition 5 characterizes the $C_{k^\ast}$-adjacency of $X\times Y$.
	For instance, while the product $SC_8^{2,4} \times SC_{26}^{3,4}\subset {\mathbb Z}^5$ has the two types of $C$-compatible $k$-adjacencies, $k \in \{130=k(3,5), 210=k(4,5)\}$, it has the only  $C_{130}$-adjacency because $130=k^\ast=k(3,5)$, where $3=\text{max}\{2,3\}$. 
	\end{remark}

By Definitions 3 and 5, the following is obtained.
\begin{remark} \cite{H10}
		Assume $(X\times Y, C_{k^\ast})$ exists. Then the product $X\times Y$ also has the $G_{k^\ast}$-adjacency such that $G_{k^\ast}=k^\ast=C_{k^\ast}$.
	\end{remark}

\begin{remark} (1) Consider $(X_i, k_i)$ in ${\mathbb Z}^{n_i}$, $i \in \{1, 2\}$, where $k_i:=k(t_i, n_i)$ (see (2.1)).
	Assume the case $k_i=2n_i$, i.e., $t_i=1, i \in \{1, 2\}$. Then a $C_{k^\ast}$-adjacency on the product $X_1 \times X_2$ always and uniquely exists, where $k^\ast:=k(1, n_1+n_2)=2(n_1+n_2)$, i.e., $C_{k^\ast}=2(n_1+n_2)=k^\ast$ \cite{H10}.\\	
		(2) In case $(X_1 \times X_2, C_{k^\ast}$) exists, based on (1) above, we need to
		  make the number $k^\ast$ of the $C_{k^\ast}$- and $G_{k^\ast}$-adjacency consistent to avoid some confusion of establishing a continuity of the multiplication $\alpha_1$ or $\alpha_1^\prime:X^2 \to X$ referred to in (5.1) and to finally define the continuity as in Definition 10 without any confusion of taking a  $C_{k^\ast}$- and $G_{k^\ast}$-adjacency on $X^2$ for the $(G_{k^\ast}, k)$- or $(C_{k^\ast}, k)$-continuity of the multiplication $\alpha_1$ or $\alpha_1^\prime$ of (5.1) (in detail, see Section 3.2 in the present paper).		  
		   Hence the paper \cite{H10} took the processes like (3.1) and Definitions 3 and 5 with the same number $k^\ast$.	
\end{remark}

Based on the $C_{k^\ast}$-adjacency,  the following $C_{k^\ast}$-neighborhood of a point of  $X_1\times X_2$ is defined.

\begin{definition} \cite{H10} Assume $(X\times Y, C_{k^\ast})$.
	For an element $p \in X \times Y$, we define the $C_{k^\ast}$-neighborhood of $p$, as follows:
	$$\left \{
\aligned  
& 	N_{C_{k^\ast}}^\ast(p):=\{q \in X \times Y\,\vert \, q \,\,\text{is}\,\,C_{k^\ast}\text{-adjacent to}\,\,p\}\,\, \text{and}\,\\
& N_{C_{k^\ast}}(p):=N_{C_{k^\ast}}^\ast(p) \cup \{p\}.
\endaligned
\right\}\eqno(3.3)  $$	
\end{definition}

Let us further compare among $C_{k^\ast}$-, $G_{k^\ast}$-, and the $k^\ast$-adjacency and their neighborhoods of a point $p$ in the given set, as follows:
\begin{corollary} \cite{H10}
	 In case  $(X\times Y,C_{k^\ast})$ exists, the following is obtained.
	$$ N_{C_{k^\ast}}(p)= N_{G_{k^\ast}}(p)= N_{k^\ast}(p) \text{(see Corollary 4.7 of \cite{H10})}.$$
\end{corollary}

		\subsection{\bf The importance of the $(C_{k^\ast}, k)$- and $(G_{k^\ast}, k)$-continuity for establishing a $DT$-$k$-group}\label{s3-2}
	
		The paper \cite{H10} defined the notion of $(G_{k^\ast}, k)$-continuity of the map $\alpha_1$ of (5.1) in Section 5, as follows: For a digital image $(X,k)$, we say that the map  $\alpha_1:(X^2, G_{k^\ast}) \to (X, k)$ is $(G_{k^\ast}, k)$-continuous if, 
	for two points $x$ and $x^\prime$ which are $G_{k^\ast}$-adjacent on $X^2$, $f(x)=f(x^\prime)$ or $f(x)$ is $k$-adjacent to $f(x^\prime)$ (see Definition 4.18 of \cite{H10}.
		
		As an equivalent presentation of the $(G_{k^\ast}, k^\prime)$-continuity, we have the following \cite{H10}:
		
		\begin{proposition} (see Proposition 4.19 of \cite{H10}) 
			A function $f: (X_1\times X_2,  G_{k^\ast}) \to (Y, k^\prime)$ is $(G_{k^\ast}, k^\prime)$-continuous  if and only if for each $p \in X_1\times X_2$
			$$f(N_{G_{k^\ast}}(p)) \subset N_{k^\prime}(f(p)). $$
		\end{proposition}

		As mentioned in Remark 3.4, since a $C_{k^\ast}$-adjacency is a restricted version of a $G_{k^\ast}$-adjacency, let us recall the notion of  $(C_{k^\ast},k)$-continuity in \cite{H10}, as follows:
	
				 If $X^2$ has a $C_{k^\ast}$-adjacency, we say that the map  $\alpha_1^\prime: (X^2,  C_{k^\ast}) \to (X, k)$ of (5.1) is  $(C_{k^\ast}, k)$-continuous if, 
			for	two points $x$ and $x^\prime$ which are $C_{k^\ast}$-adjacent on $X^2$, $f(x)=f(x^\prime)$ or $f(x)$ is $k$-adjacent to $f(x^\prime)$ (see Definition 3.8 in \cite{H10}).
		As an equivalent presentation of the $(C_{k^\ast}, k^\prime)$-continuity, the following is obtained \cite{H10}:
	
	\begin{proposition} (see Proposition 3.9 of \cite{H10}) For digital images $(X_i, k_i)$, $i \in \{1,2\}$ and $(Y, k^\prime)$, a function $f: (X_1\times X_2,  C_{k^\ast}) \to (Y, k^\prime)$ is $(C_{k^\ast}, k^\prime)$-continuous  if and only if for each $p \in X_1\times X_2$
		$$f(N_{C_{k^\ast}}(p)) \subset N_{k^\prime}(f(p)). $$
	\end{proposition}

		\section{\bf Improvement of the $NP_u(k_1, \cdots, k_v)$-adjacency and benefits of the $AP_u(k_1, \cdots, k_v)$-adjacency for a  product of  multiple digital images}\label{s4}

	This section first takes issues with unclearness on the $NP_u(k_1, \cdots, k_v)$-adjacency referred to in \cite{B1,B2} (see Definition 7 in the present paper and Remark 4.1). Hence we will be using another approach to improve it.
	Before going there, in relation to the notations of $k$-adjacencies of ${\mathbb Z}^n$ (see (2.1)), as a convenience, as mentioned in Section 2, let us point out some relationships between Boxer's and Han's approach to the digital adjacencies of ${\mathbb Z}^n$ as follows.
	Based on the criterion for establishing the $k$-(or $k(t,n)$-)adjacencies of  ${\mathbb Z}^n$ of (2.1), the paper \cite{B1} takes the notation $\lq\lq$$c_t$-adjacency on ${\mathbb Z}^n$", $t \in [1,n]_{\mathbb Z}$, with a just different notation. Indeed, it is same as the $\lq\lq$$k(t, n)$-adjacency" of (2.1), $t \in [1,n]_{\mathbb Z}$ \cite{H1,H2,H4}, i.e., the $\lq\lq$$c_t$-adjacency on ${\mathbb Z}^n$" of \cite{B1} is equal to the $\lq\lq$$k(t, n)$-adjacency" of (2.1). \\
	As a generalization of the adjacencies of Definitions 1 and 2 (see Proposition 4.5), the paper \cite{B2} referred to the so-called $NP_u(k_1, \cdots, k_v)$-adjacency for the product $\prod_1^v X_i \subset {\mathbb Z}^{n_1+n_2+ \dots +n_v}$ from  the $v$ multiple digital images $(X_i, k_i:=k(t_i, n_i))$, $t_i\in [1,n_i]_{\mathbb Z}$, $X_i \subset {\mathbb Z}^{n_i}, i \in [1, v]_{\mathbb Z}$.\\
	
		To make the paper self-contained, let us now recall the $NP_u(k_1, \cdots, k_v)$-adjacency mentioned in \cite{B1}, as follows:
						
	\begin{definition} (Definition 3.1 of \cite{B1}) 
		Assume the $v$ multiple digital images $(X_i, k_i:=k(t_i, n_i))$, $t_i\in [1,n_i]_{\mathbb Z}$, $X_i \subset {\mathbb Z}^{n_i}, i \in [1, v]_{\mathbb Z}$ (or $(X_i, c_{t_i})$ on ${\mathbb Z}^{n_i}$ according to the notation in \cite{B1}). Then consider the  product $\prod_1^v X_i \subset {\mathbb Z}^{n_1+n_2+ \dots +n_v}$.
		Given positive integers $v$ and $u \in [1,v]_{\mathbb Z}$, let $NP_u(k_1, \cdots , k_v)$ be the adjacency defined on the product $\prod_{1}^v X_i$ as follows. \\
		For $x_i, x_i^\prime \in X_i$, $p=(x_1, \cdots, x_v)$ and $q=(x_1^\prime, \cdots , x_v^\prime)$ are 
		$NP_u(k_1, \cdots, k_v)$-adjacent if and only if	
		$$\left \{
	\aligned
	& (\bullet\,1)\,\text{for at least}\,\,1 \,\,\text{and at most}\,\,u\,\,\text{indices}\,\,i, x_i\,\,\text{and}\,\,x_i^\prime\,\,\text{are}\,\,k_i\text{-adjacent, and}\\
   &(\bullet\,2)\,\,\text{for all other indices}\,\,i, x_i = x_i^\prime.	
	\endaligned
	\right\}\eqno(4.1)  $$		
	\end{definition}
	
	Owing to $(\bullet\,1)$ of (4.1), we note that the given points $p$ and $q$ in Definition 7 are distinct from each other.
	It is clear that  Definition 7 has some unclearness from the viewpoint of logic. To be specific, let us focus on the part (see from line 5  to 8 in Definition 7) $\lq\lq$ let $NP_u(k_1, \cdots, k_v)$ be the adjacency defined on the product $\prod_{1}^v X_i$ as follows \cite{H12}. For $x_i, x_i^\prime \in X_i$,	
	$$\left \{
	\aligned
	& p=(x_1, \cdots, x_v)\,\,\text{and}\,\,q=(x_1^\prime, \cdots , x_v^\prime)\,\,\text{are} \,\,
	NP_u(k_1, \cdots, k_v)\text{-adjacent}\\
	& \Leftrightarrow\\
	& \text{the two conditioins}\,\,(\bullet\,1)\,\,\text{and}\,\,(\bullet\,2)\,\,\text{of Definition 7 hold.}"
		\endaligned
	\right\}\eqno(4.2)  $$	
	
	To show the invalidity of the statement of (4.2), to make the paper self-contained, we now recall the following:
	As used in Definitions 1 and 2, in order to establish a certain definition (or notion), we can use the term $\lq\lq$ if and only if". For instance, given two propositions $P$ and $Q$, $P$ if and only if $Q$,
	Then it is clear that each of $P$ and $Q$ should be a proposition.
	In this respect, we can clearly see that Definition 7 does not have this kind of approach (see Remark 4.1).	
		\begin{remark} \cite{H12} 
  In order for the sentence of (4.2) to be valid as a statement (or proposition) from the viewpoint of logic, the condition $\lq\lq$$p$ and $q$ are $NP_u(k_1, \cdots, k_v)$-adjacent" should be a statement. However, it does not defined, i.e., one does not know what it is. After comparing the $k$-adjacency for establishing the normal $k$-adjacency of Definition 1 and the $NP_u(k_1, \cdots, k_v)$-adjacency of Definition 7, we can see some defects of the latter.
 In detail, the term $NP_u(k_1, \cdots, k_v)$ was not identified before using it as a statement for (4.1) (see also (4.2)).
	Indeed, many kinds of adjacencies on $\prod_{1}^v X_i$ can be considered, e.g., Definitions 1, 2, 3, and 4 and the $(n_1+n_2+ \dots +n_v)$-types of adjacencies of (2.1) and so on. Namely, one does not know about the $NP_u(k_1, \cdots, k_v)$-adjacency so that it is insufficient for an adjacency for the statement of (4.2), which implies that the sentence 
	$\lq\lq$$p$ and $q$ are $NP_u(k_1, \cdots, k_v)$-adjacent $\Rightarrow$ the two conditions $(\bullet\,1)$ and  $(\bullet\,2)$ of Definition 7 hold" could not be admissible as a proposition.\\
	Similarly, the part $\lq\lq$$\Leftarrow$" cannot be valid either.	
		\end{remark}
	
			Probably, when formulating the $NP_u(k_1,k_2, \cdots, k_v)$-adjacency \cite{B1,B2}, there seems to be a mismatch between the author's intention and the mathematical presentation.\\	
				
			In order to avoid some defects of the $NP_u(k_1, \cdots, k_v)$-adjacency of \cite{B1} mentioned just before (see Definition 7 and Remark 4.1), the paper \cite{B2} defined the $NP_u(k_1, \cdots, k_v)$-adjacency again by just replacing $\lq\lq$if and only if" used in Definition 7 by $\lq\lq$if" and the remaining part is fixed to avoid the roughness of Definition 7. However, it also has some fatal drawbacks. To be specific, for instance, let us consider two			 
			digital images $(X_1,4)$ and $(X_2,8)$ on ${\mathbb Z}^2$, where $X_1:=\{x_1=(0,0), x_2=(1,0), x_3=(1,1)\} $ and $X_2:=\{x_1^\prime=(0,0), x_2^\prime=(1,1)\} $.
			Take the point $p:=(x_1,x_1^\prime) \in X_1 \times X_2 \subset {\mathbb Z}^4$ and consider the following two sets around the point $p$, $T_p$ and $T_p^\prime$, i.e.,
				$$
				\begin{cases}
				&T_p= \{q\,\vert \,p\,\,\text{and}\,\,q\,\,\text{satisfy the statement of (4.1) with}\,\, u=1, v=2.\} \,\,\text{and}\\
				&T_p^\prime= \{q\,\vert \,p\,\,\text{and}\,\,q\,\,\text{satisfy the statement of (4.1) with}\,\, u=2, v=2.\}
				\end{cases} $$	
				
				Then we note that the element $(x_3,x_1^\prime) \notin  T_p^\prime=\{(x_2,x_1^\prime), (x_1,x_2^\prime), (x_2,x_2^\prime)\}$.
				Besides, it is clear that $T_p=\{(x_2,x_1^\prime), (x_1,x_2^\prime)\} \subset X_1\times X_2 \subset {\mathbb Z}^4$ with only two elements.
			
			Meanwhile, let us consider the subsets of $X_1\times X_2$.
			
			$$
			\begin{cases}
			& N_8^\ast(p)=\{(x_2,x_1^\prime)\}\,\,\text{as a singleton},\\
			& N_{32}^\ast(p)=\{(x_1,x_2^\prime),(x_2,x_1^\prime), (x_3,x_1^\prime)\} \,\,\text{with three elements},\\
			& N_{64}^\ast(p)= \{(x_1,x_2^\prime),(x_2,x_1^\prime), (x_3,x_1^\prime),(x_2,x_2^\prime)\}\,\, \text{with four points, and}\\
			& N_{80}^\ast(p)= X_1 \times X_2 \setminus \{p\}\,\,\text{with five elements}.
			\end{cases} $$	
			
			Thus we can obviously confirm that each of the sets $T_p$ and $T_p^\prime$ is not equal to any of $N_k^\ast(p)$, where $k \in\{8,32,64,80\}$, which implies that
			the relation satisfying the statement of (4.1) cannot be a digital $k$-adjacency of ${\mathbb Z}^4$. In particular, we need to note that  $T_p^\prime$ is not equal to either of
			$N_{32}^\ast(p)$ and $N_{64}^\ast(p)$ due to $(x_3,x_1^\prime) \notin T_p^\prime$.			
			Finally, in general, the adjacency on the product $\prod_{1}^v X_i$
			satisfying  the statement of (4.1) need not be a digital $k$-adjacency of ${\mathbb Z}^{n_1+n_2+ \dots +n_v}$ of (2.1).
			Indeed, it turns out that the
			$NP_u(k_1, k_2, \cdots, k_v)$-adjacency of \cite{B2} is exactly equal to the $PN(k_1, k_2, \cdots, k_v)$-adjacency of the product  $\prod_{1}^v X_i$ that is a generalized version of the $PN(k_1, k_2)$-adjacency of the product  $X_1\times X_2$ of \cite{KHL1}.	\\
			
			Anyway, since each of all trials in \cite{B1,B2} will be failed to formulate the 
			$NP_u(k_1, k_2, \cdots, k_v)$-adjacency, all results having used the  $NP_u(k_1,k_2, \cdots, k_v)$-adjacency should be improved. Hence the recent paper \cite{H12} initially corrected the $NP_u(k_1,k_2, \cdots, k_v)$-adjacency (see Definition 8 below) and improve it with an $AP_u(k_1,k_2, \cdots, k_v)$-adjacency for a product of $v$ multiple digital images, $u \in [1,v]_{\mathbb Z}$ (see Definition 8), where the term $\lq\lq$$AP$" of 
			the $AP_u(k_1,k_2, \cdots, k_v)$-adjacency is an abbreviated word of $\lq\lq$Adjacency for the Product $\prod_{1}^v X_i$", as follows:
			:\\

					\begin{definition} (Correction of Definiton 7) 
					In Definition 7, replace the part with lines 5-10 			
					 $\lq\lq$Given positive integers $v$ and $u \in [1,v]_{\mathbb Z}$, 
					 consider the  product $\prod_1^v X_i \subset {\mathbb Z}^{n_1+n_2+ \dots +n_v}$.  For $x_i, x_i^\prime \in X_i$, $p=(x_1, \cdots, x_v)$ and $q=(x_1^\prime, \cdots , x_v^\prime)$ are $k$-adjacency of  ${\mathbb Z}^{n_1+n_2+ \dots +n_v}$ of (2.1) if and only if both the conditions 
					 $(\bullet 1)$ and $(\bullet 2)$ hold.
					  Then the $k$-adjacency is called an $NP_u(k_1, \cdots , k_v)$-adjacency (or $AP_u(k_1, \cdots , k_v)$-adjacency) on the product $\prod_{1}^v X_i$".\\
					 	In particular, put 
					 $$k^\ast:=\text{min}\{AP_u(k_1, k_2, \cdots, k_v)\text{-adjacencies of}\,\, \prod_1^v X_i\}.$$
					 Then we say that the $k^\ast$-adjacency is an $AP_u^\ast(k_1, k_2, \cdots, k_v)$-adjacency on $\prod_1^v X_i$, $u \in [1,v]_{\mathbb Z}$.
					 
				\end{definition}	
						Hereinafter, since Definition 8 is a generalization of Definitions 1 and 2, we had better call the adjacency of Definition 8 an $AP_u(k_1,k_2, \cdots, k_v)$-adjacency instead of the  $NP_u(k_1, \cdots , k_v)$-adjacency.
						
			\begin{proposition} \cite{H12} For the $v$ multiple digital images $(X_i, k_i:=k(t_i, n_i))$,  
				$X_i \subset {\mathbb Z}^{n_i},  t_i\in [1,n_i]_{\mathbb Z}, i \in [1, v]_{\mathbb Z}$, consider the  product $\prod_1^v X_i \subset {\mathbb Z}^{n_1+n_2+ \dots +n_v}$ and a $k$-adjacency of  ${\mathbb Z}^{n_1+n_2+ \dots +n_v}$ of (2.1).
				Assume that each point $p:=(x_1,x_2, \cdots, x_v) \in \prod_1^v X_i$ satisfies the following:
							$$\left \{
				\aligned
				& N_k(p)=A_1 \times A_2 \times \cdots \times A_v, A_i \subset X_i, i\in [1,v]_{\mathbb Z},\,\, \text{such that}	\\			
				& (\star 1)\, \text{for at least}\,\,1\,\,\text{and at most}\,\,u \,\,\text{indices}\,\,i, A_i=N_{k_i}(x_i)\,\,\text{and}\,\\
				&(\star 2)\,\text{for all other indices}\,\,i, A_i = \{x_i\}. 
				\endaligned
				\right\}\eqno(4.3)  $$
				Then the $k$-adjacency is an $AP_u(k_1, k_2, \cdots, k_v)$-adjacency on $\prod_1^v X_i$, $u \in [1,v]_{\mathbb Z}$.
						\end{proposition}	
				\begin{proof} Using the notion of $N_{k_i}(x_i)$ and Definition 8, the proof is completed.
				\end{proof} 
				Besides, the given points $p$ and $q$ in Proposition 4.2 are distinct from each other (see ($(\star\,1)$) of (4.3)).
							The reader should be careful to make a distinction between Definitions 7 and 8 (see Remark 4.1). Besides, the following are observed.\\				
				$(\ast)$ As for the number $k^\ast$ of Definition 8, the following is obtained.
				$$k^\ast \leq AP_u(k_1,k_2, \cdots, k_v). $$
				$(\ast)$ Each of the adjacencies $AP_u(k_1, k_2, \cdots, k_v)$ and $AP_u^\ast(k_1, k_2, \cdots, k_v)$ is one the adjacencies of ${\mathbb Z}^{n_1+n_2+ \dots +n_v}$ mentioned in (2.1). \\
				$(\ast)$ Not every product $\prod_{1}^v X_i$ of Definition 8 always has an 
				$AP_u(k_1, k_2, \cdots, k_v)$-adjacency (see Remark 4.4).\\
				$(\ast)$ In particular, in Definition 8, if $k_i=k_j, i, j \in [1,v]_{\mathbb Z}$, then we will take the notation  $AP_u(k_1, k_1, \cdots, k_1)$ (resp. $AP_u^\ast(k_1, k_1, \cdots, k_1))$ instead of $AP_u(k_1, k_2, \cdots, k_v)$ (resp. $AP_u^\ast(k_1, k_2, \cdots, k_v))$.	
				
			\begin{example} Let $X:=\{x_0=(0,0,0), x_1=(1,0,1), x_2=(2,1,1), x_3=(1,2,1), x_4=(0,2,0), x_5=(-1,1,0)\} $ and assume the digital image $(X, 18)$. Indeed, this $(X, 18)$ is a kind of $SC_{18}^{3,6}$ (see Section 2).\\
		(1)	Then we obtain the adjacencies $AP_u(18,18)$ and $AP_u^\ast(18,18)$ on the product $X^2\subset {\mathbb Z}^6$, as follows:
						$$
			\begin{cases}
			& AP_1(18,18) \in \{k(2,6), k(3,6)\}, \,\,AP_1^\ast(18,18)=k(2,6) \,\,\text{and}\,\, \\
			& AP_2(18,18) \in \{k(4,6), k(5,6), k(6,6)\}, \,\,AP_2^\ast(18,18)=k(4,6).
			\end{cases} $$
			
			(2) Let us consider $X^3\subset {\mathbb Z}^9$ with the adjacencies	$AP_u(18,18,18)$ and $AP_u^\ast(18,18,18)$. Then we obtain the following:
					$$
			\begin{cases}
			& AP_1(18,18,18) \in \{k(2,9), k(3,9)\}, \,\,AP_1^\ast(18,18,18)=k(2,9),  \\
			& AP_2(18,18,18) \in \{k(4,9), k(5,9)\}, \,\,AP_2^\ast(18,18,18)=k(4,9)\,\,\text{and}\,\,\\
			& AP_3(18,18,18) \in \{k(6,9),k(7,9), k(8,9), k(9,9)\}, \,\,AP_3^\ast(18,18,18)=k(6,9).
						\end{cases} $$
									\end{example}

\begin{remark} With Definition 8, not every  $\prod_1^v X_i$ always has an $AP_u(k_1, k_2, \cdots, k_v)$-adjacency, as follows:\\	
 (1)  $SC_{2n_1}^{n_1,l_1} \times SC_{2n_2}^{n_2,l_2} \times SC_{2n_3}^{n_3,l_3} \subset {\mathbb Z}^{n_1+n_2+n_3}$ has  an $AP_1(2n_1,2n_2, 2n_3)$-adjacency, where 
 $$AP_1(2n_1,2n_2, 2n_3)=k(1,n_1+n_2+n_3)=AP_1^\ast(2n_1,2n_2, 2n_3)=2(n_1+n_2+n_3).$$
   However, it does not have any $AP_u(2n_1,2n_2, 2n_3)$-adjacency, $u \in \{2,3\}$.\\
(2) For the product $SC_8^{2,4} \times SC_8^{2,6} \times SC_8^{2,4}\subset {\mathbb Z}^6$,
we have the adjacency $AP_2(8,8,8)\in\{k(4,6), k(5,6)\}$ \cite{H12} so that $AP_2^\ast(8,8,8)=k(4,6)$. Besides, we obtain $AP_3(8,8,8)=k(6,6)=AP_3^\ast(8,8,8)=728$-adjacency (see (2.2)).\\ 
	(3) Assume	
	 $SC_{2n_1}^{n_1,l_1} \times SC_{k_2}^{n_2,l_2} \times SC_{2n_3}^{n_3,l_3} \subset {\mathbb Z}^{n_1+n_2+n_3}$ such that $k_2\neq 2n_2$.
	 Then it does not have any  $AP_u(2n_1,k_2,2n_3)$-adjacency, $u\in [1,3]_{\mathbb Z}$.
	 For instance, $SC_4^{2,8} \times SC_8^{2,6}\times SC_4^{2,8} \subset {\mathbb Z}^6$
	 does not have any $AP_u(4,8,4)$-adjacency, $u \in [1,3]_{\mathbb Z}$. 
\end{remark}  

After comparing Definitions 7 and 8, based on Remark 4.4, we clearly obtain the following:
\begin{proposition} (1) An $AP_1(k_1, \cdots, k_v)$-adjacency  is independent of an $AP_u(k_1, \cdots, k_v)$-adjacency, $u\neq 1$.\\
	(2) An $AP_1(k_1, k_2)$-adjacency is equivalent to the $C$-compatible $k$-adjacency of Definition 2 \cite{H12}. Besides, an $AP_1^\ast(k_1, k_2)$-adjacency is equivalent to the $C_{k^\ast}$-adjacency of Definition 5. \\
	(3) An $AP_2(k_1,k_2)$-adjacency is equivalent to the normal $k$-adjacency of Definition 1 \cite{H12}. Besides, for two digital images $(X_i,k_i), i \in \{1,2\}$, an $AP_2^\ast(k_1,k_2)$-adjacency on $X_1\times X_2$ implies a normal $k$-adjacency of the product. However, a normal $k$-adjacency need not be an $AP_2^\ast(k_1,k_2)$-adjacency. \\
	(4) For $(X,k_1:=k(t_1,n_1))$ and $(Y,k_2:=k(t_2,n_2))$, in case $(X\times Y, AP_2(k_1,k_2))$ exists, the $(3^{n_1+n_2}-1)$-adjacency need not coincide with the $AP_2(k_1,k_2)$-adjacency. Namely, $AP_2(k_1,k_2)$ need not equal to the number
	$3^{n_1+n_2}-1$.
	\end{proposition}
To make the paper self-contained, we now prove it briefly.
\begin{proof} (1) Based on Remark 4.4, the proof is completed.
	To be specific, in case $i\neq j, i, j \in  [1, v]_{\mathbb Z}$, none of the $AP_i(k_1,k_2, \cdots, k_v)$-adjacency and $AP_j(k_1,k_2, \cdots, k_v)$-adjacency implies the other, in general.\\	
	(2) In case $(X_1\times X_2, AP_1(k_1,k_2))$ exists, owing to the notions of $AP_1(k_1,k_2)$-adjacency and $C$-compatible $k$-adjacency, we obtain 
	an $AP_1(k_1,k_2)$-adjacency is equivalent to a $C$-compatible $k$-adjacency. 
	Namely, for a point $p:=(x_1,x_2) \in X_1\times X_2$, we obtain the adjacency $k^\ast$ of Definition 5	
	such that
	$$N_{AP_1^\ast(k_1,k_2)}(p)=N_{k^\ast}(p)=N_{C_{k^\ast}}(p), $$
	which completes the proof.\\
	For instance, it is clear that the product $SC_8^{2,4} \times SC_8^{2,4} \subset {\mathbb Z}^4$ has an $AP_1(8,8)$-adjacency, where $AP_1(8,8)\in\{32, 64\}$ and $AP_1^\ast(8,8)=32=C_{32}$.	\\	
	(3) In case $(X_1\times X_2, AP_2(k_1,k_2))$ exists, we obtain a normal $k$-adjacency on $X_1\times X_2$ such that
	 $k=AP_2(k_1,k_2)$.
	  However, the converse need not hold.
	 For instance, while $SC_8^{2,6} \times SC_{18}^{3,6} \subset {\mathbb Z}^5$ has two types of normal $k$-adjacencies, where $k \in \{k(4,5), k(5,5)\}$ and $SC_{18}^{3,6}:=X$ in Example 4.3, it has only the $AP_2^\ast(8,18)$-adjacency, $AP_2^\ast(8,18)=k(4,5)=210$ (see (2.2)).\\
	 (4) Owing to the product $SC_8^{2,6} \times SC_{18}^{3,6} \subset {\mathbb Z}^5$  just mentioned in (3) above, the proof is completed.	
	\end{proof}

\begin{example}
	(1) $SC_8^{2,4} \times SC_8^{2,4} \times SC_8^{2,4} \subset {\mathbb Z}^6$ has the $AP_1^\ast(8,8,8)$-adjacency, where $AP_1^\ast(8,8,8)=k(2,6)$. 
	Besides, it has the $AP_2^\ast(8,8,8)$-adjacency, $AP_2^\ast(8,8,8)=k(4,6)=472$ and it also has $AP_3^\ast(8,8,8)=k(6,6)$.\\
		(2)	$SC_8^{2,6} \times SC_8^{2,6} \times SC_6^{3,6} \subset {\mathbb Z}^7$ does not have any $AP_u(8,8,8)$-adjacency, $u \in [1,3]_{\mathbb Z}$. \\	
	(3) $SC_8^{2,4} \times SC_{8}^{2,4} \subset {\mathbb Z}^4$ has both the $AP_1^\ast(8,8)$-adjacency, where $AP_1^\ast(8,8)=k(2,4)$ and the $AP_2^\ast(8,8)$-adjacency, where $AP_2^\ast(8,8)=k(4,4)$.
\end{example}

Motivated by Remark 4.4(1), the following is obtained.

\begin{lemma} Assume the set ${\mathbb Z}^n, 2\leq n=n_1+n_2+\cdots, n_v, n_i \in {\mathbb N}, i \in [1,v]_{\mathbb Z}$.
	Then ${\mathbb Z}^n$ has an $AP_1(2n_1,2n_2, \cdots, 2n_v)$-adjacency derived from $({\mathbb Z}^{n_i}, 2n_i), i \in [1,v]_{\mathbb Z}$.
	Then we obtain $AP_1(2n_1, 2n_2, \cdots, 2n_v)=G_{2n}=C_{2n}=2n=AP_1^\ast(2n_1,2n_2, \cdots, 2n_v)$.
\end{lemma}

\begin{definition} Given an $AP_u:=AP_u(k_1,k_2, \cdots,  k_v)$-adjacency on the product $\prod_1^v X_i \subset {\mathbb Z}^{n_1+n_2+ \dots +n_v}$,	
 for $p \in \prod_1^v X_i, u \in [1,v]_{\mathbb Z}$, we define
	$$N_{AP_u}^\ast(p):=\{q \in \prod_1^v X_i \,\vert \, q \,\,\text{is}\,\,AP_u\text{-adjacent to}\,\,p\} \eqno(4.4)$$
	and
	$$N_{AP_u}(p):=N_{AP_u}^\ast(p) \cup \{p\},\eqno(4.5)$$
	which is called an $AP_u$-neighborhood of $p$ in  $\prod_1^v X_i$.\\
	Besides,  we define
	$$N_{AP_u^\ast}^\ast(p):=\{q \in \prod_1^v X_i\,\vert \, q \,\,\text{is}\,\,AP_u^\ast\text{-adjacent to}\,\,p\} \eqno(4.6)$$
	and
	$$N_{AP_u^\ast}(p):=N_{AP_u^\ast}^\ast(p) \cup \{p\},\eqno(4.7)$$
	which is called an $AP_u^\ast$-neighborhood of $p$ in  $\prod_1^v X_i$.\\
\end{definition}

In view of (4.4) and (4.6), the following is obtained.
\begin{corollary} (1) In case $(X_1\times X_2, AP_1:=AP_1(k_1,k_2))$ exists,
	  for $p:=(x_1, x_2) \in X_1\times X_2$, we have
	$$N_{AP_1}(p)=(N_{k_1}(x_1) \times \{x_2\}) \cup  (\{x_1\}\times N_{k_2}(x_2))= N_{AP_1^\ast}(p). \eqno(4.8)$$
	
	(2) In case $(X_1\times X_2, AP_2:=AP_2(k_1,k_2))$ exists, we obtain
	 $$N_{AP_2}(p)=N_{k_1}(x_1)\times N_{k_2}(x_2)=N_{AP_2^\ast}(p).$$ 
\end{corollary}

	\begin{proposition} Not every $SC_k^{n,l}\times SC_k^{n,l}$ has an $AP_u(k,k)$-adjacency, $u \in \{1,2\}$.\end{proposition}
\begin{proof} To invoke a contradiction, consider $MSC_{18}$ of (2.9) which can obviously be assumed to be an $SC_{18}^{3,6}$. Then, by Corollary 4.8, we see that the product $MSC_{18}\times MSC_{18}$ does not have any $AP_u(18,18), u \in \{1,2\}$ (see also Theorem 2.6).
	\end{proof}

\section{\bf Comparison between two types of $DT$-group structures and some improvements}\label{s5}

Based on the literature of a $DT$-group structure, let us now recall two kinds of approaches for formulating a $DT$-group structure on $(X,k)$ mentioned in \cite{H10,LS1} (see Types 1 and 2 below), as follows: \\

(Type 1) Han's $DT$-$k$-group structure \cite{H10} using the $G_{k^\ast}$- or $C_{k^\ast}$-adjacency on $X^2$ in \cite{H10}:\\
Given a digital image $(X,k)$, to establish a $DT$-$k$-group structure, the $G_{k^\ast}$- or $C_{k^\ast}$-adjacency on the product $X^2$ \cite{H10} was considered, i.e., a digital image $(X^2, C_{k^\ast})$ and a digital space $(X^2, G_{k^\ast})$ as relation sets that are both non-reflexive and symmetric \cite{H10} (see Definitions 3 and 5). 
Indeed, they play important roles in establishing a $DT$-$k$-group structure on $(X,k)$ in \cite{H10}.
And that, the ($G_{k^\ast}, k)$- or ($C_{k^\ast},k)$-continuity for the multiplication $\alpha_1$ of (5.1) below was also considered in \cite{H10} (see Definition 5.5 and Corollary 5.7 of \cite{H10} such as Definition 10 below and Remark 5.1) 
$$\alpha_1:(X^2, G_{k^\ast})\to (X, k)\,\, \text{or}\,\,\alpha_1^\prime:(X^2, C_{k^\ast})\to (X, k). \eqno(5.1)$$
It is clear that the map $\alpha_1^\prime$ is a special case of $\alpha_1$ because a $C_{k^\ast}$-adjacency implies a $G_{k^\ast}$-adjacency (see Remark 3.5). 
Namely, while $\alpha_1^\prime$ implies $\alpha_1$, the converse does not hold in general.  \\

After developing an $AP_1$-$k$-group structure of $(X,k)$ in Section 6 later, we will compare the $DT$-$k$-group in \cite{H10} with the $AP_1$-$k$-group (see Theorem 6.9). Hence let us first recall the $DT$-$k$-group structure in \cite{H10}.

\begin{definition} \cite{H10}
	Assume both a digital image $(X,k)$ and a group $(X, \ast)$. 
	We say that the set $(X,k,\ast)$ is a  digital-topological ($DT$-, for brevity) $k$-group if, for $(x,y) \in X^2$, the multiplication
	$$ \alpha_1: (X^2, G_{k^\ast}) \to (X, k)\,\,\text{defined by}\,\,\alpha_1(x, y)= x \ast y\,\,\text{is}\,\,(G_{k^\ast}, k)\text{-continuous} \eqno(5.2)$$
	and the inverse map with respect to the operation $\ast$
	$$\beta: (X,k) \to (X,k)\,\,\text{defined by}\,\,\beta(x)= x^{-1}\,\,\text{is}\,\,k\text{-continuous}. \eqno(5.3)$$
\end{definition}

In Definition 10, for a  $DT$-$k$-group $(X, k, \ast)$, in case the given group  $(X,\ast)$ is abelian, we say that this $DT$-$k$-group is abelian \cite{H10}.
As for the $G_{k^\ast}$-adjacency of $X^2$ in Definition 10, we strongly need to recall the following:

\begin{remark} (1) As for the condition (5.2) of Definition 10, in only the case of $(X^2, C_{k^\ast})$ derived from $(X, k)$, the $C_{k^\ast}$-adjacency
	on $X^2$ implies the $G_{k^\ast}$-adjacency on $X^2$ (see Remark 3.4). Hence, in case $(X^2, C_{k^\ast})$ exists, the phrase $\lq\lq$$(G_{k^\ast}, k)$-continuous" of (5.2) can be replaced by $\lq$$(C_{k^\ast}, k)$-continuous' \cite{H10} (see Proposition 3.8) for another continuity.\\
	(2) Given $(X^2, C_{k^\ast})$ (see Definition 5), the term  $\lq$$(C_{k^\ast}, k)$-continuous' is equivalent to  $\lq$$(k^\ast, k)$-continuous'. Then the $DT$-$k$-group is also established by using the $(C_{k^\ast}, k)$-continuity for the multiplication $\alpha_1^\prime:(X^2, C_{k^\ast}) \to (X,k)$ and the map $\beta$ of (5.3).\\
	(3) Owing to (1) and (2) above, the $DT$-$k$-group using  the $\lq\lq$$(C_{k^\ast}, k)$-continuity" of (5.2) is stronger than the 
	the $DT$-$k$-group using  the $\lq\lq$$(G_{k^\ast}, k)$-continuity".	
\end{remark}

When studying the $DT$-group structure of Definition 10, we should remind the reason why we take the number(subscript) $k^\ast$ of the $G_{k^\ast}$- and $C_{k^\ast}$-adjacency with the same number (see Remarks 3.3, 3.4, and 3.5(2) in the present paper). \\

We should note the usage of the $G_{k^\ast}$- or $C_{k^\ast}$-adjacency on the product $X^2$ as mentioned in Remark 5.1. Namely, by contrast, suppose that we take a $C$-compatible $k^\prime$-adjacency or a normal $k^\prime$-adjacency on the product $X^2$. Then we note that not every  $(SC_k^{n,l}, \ast)$ (see also Remark 2.1) is always a $DT$-$k$-group. 
For instance, assume the digital image $X:=MSC_{18} \subset{\mathbb Z}^3$ \cite{H4} with an $18$-adjacency (see (2.9)), called the minimal simple closed $18$-curve with six elements in ${\mathbb Z}^3$. Indeed, it is not $18$-contractible but $26$-contractible \cite{H4} (see the proof of Theorem 4.6 of \cite{H4}).  
Indeed, $MSC_{18}$ is a kind of $SC_{18}^{3,6}$ (see Section 2). Then it is clear that it cannot be a $DT$-$18$-group.
More precisely, in case we try to establish a  $DT$-$18$-group on $MSC_{18}$ by using a $C$-compatible $k$- or a normal $k$-adjacency on the product $MSC_{18} \times MSC_{18} \subset {\mathbb Z}^6$, we should meet an obstacle of establishing the $C$-compatible $k$-adjacency and the normal $k$-adjacency on $MSC_{18} \times MSC_{18} \subset {\mathbb Z}^6$, $k \in \{12,72,232, 472,664, 728\}$ (see (2.1) and (2.2)) (see also Theorem 2.6). 
Meanwhile,  $MSC_{18} \times MSC_{18}$ has the  $G_{k^\ast}$-adjacency, where $k^\ast=k(2,6)=72$.\\

As mentioned in Theorem 2.6, since not every digital image $(X,k)$ always has a $C$-compatible $k^\prime$-adjacency or a normal $k^\prime$-adjacency on the product $X \times X$, the paper \cite{H10} took the $G_{k^\ast}$-adjacency for the product $X^2$ to finally formulate a $DT$-$k$-group on $(X,k)$.
This is one of the reasons why the paper \cite{H10} took the $G_{k^\ast}$-adjacency on $X^2$. Namely, given $(X,k)$, the paper \cite{H10} took advantage of the $G_{k^\ast}$-adjacency on the product $X^2$ to establish a $DT$-$k$-group on $(X,k)$ (see Proposition 5.2). 
Thus we can confirm that the $G_{k^\ast}$-adjacency on a product gives us lots of benefits to establish a $DT$-$k$-group for any digital image $(X, k)$. 

\begin{proposition} \cite{H10} Given $(X,k)$, a Han's $DT$-$k$-group structure on $(X, k)$ always and uniquely exists.
\end{proposition}

(Type 2)  The $NP_i$-$DT$-groups \cite{LS1}, $i \in \{1,2\}$, using the $NP_i(k,k)$-adjacency on $X^2$  discussed in \cite{B1}:\\
Given $(X,k)$, $NP_i$-$DT$-groups, $ i \in \{1,2\}$, were discussed in \cite{LS1} by using the $NP_i(k,k)$-adjacency for the product $X^2$ in \cite{B1} that is a special case of the $NP_u(k_1,k_2, \cdots, k_v)$-adjacency in \cite{B1} (see Definition 7 in the present paper), $u \in [1,v]_{\mathbb Z}$.
Besides, the $(NP_i(k,k),k)$-continuities of a multiplication $\alpha_2$ (see (5.4) below) for the establishment of the $NP_i$-$DT$-groups were considered in \cite{LS1}, $i \in \{1,2\}$, as follows:
$$\alpha_2:(X^2, NP_i(k,k))\to (X, k),\,\,i \in \{1,2\}. \eqno(5.4)$$

In addition, the $NP_u(k_1,k_2, \cdots, k_v)$-adjacency, $u \in [1,v]_{\mathbb Z}$, has been used for examining the product property of the two $NP_i$-$DT$-groups in \cite{LS1}, $i \in \{1,2\}$.\\

Indeed, the $NP_u(k_1,k_2, \cdots, k_v)$-adjacency in \cite{B1} was motivated by both a $C$-compatible $k$-adjacency in \cite{H5}  and a normal $k$-adjacency established in \cite{H2}. 
However, as mentioned in Remark 4.1, it turns out that the  $NP_u(k_1,k_2, \cdots, k_v)$-adjacency in \cite{B1,B2} has some unclearness.
Due to the unclearness of the $NP_u(k_1,k_2, \cdots, k_v)$-adjacency in \cite{B1,B2} (see Remark 4.1), given a digital image $(X, k)$, 
we may consider another approach to establish a $DT$-group (see Section 6) using Proposition 4.2. \\

\section{\bf Development of an $AP_1$-$k$- and  an $AP_1^\ast$-$k$-group structure and non-product property of $AP_1$-$k_i$-groups, $i \in \{1,2\}$}\label{s6}

This section initially develops both an $AP_1$-$k$- and an $AP_1^\ast$-$k$-group structure on a digital image $(X, k)$ as digital $k$-groups. In
particular, it proves the non-product property of the $AP_1$-$k_i$-groups, $i \in \{1,2\}$.
Based on Definition 8 and Proposition 4.2, we define the following:

\begin{definition} (1) Assume $(\prod_1^v X_i, AP_u(k_1, \cdots, k_v))$ derived from the $v$ multiple digital images $(X_i, k_i:=k(t_i, n_i))$,  
	$X_i \subset {\mathbb Z}^{n_i},  t_i\in [1,n_i]_{\mathbb Z}, i \in [1, v]_{\mathbb Z}$. 
	For  $u \in [1, v]_{\mathbb Z}$, we say that the map  
	$f: (\prod_1^v X_i, AP_u(k_1, \cdots, k_v)) \to (Y,k)$
	 is  $(AP_u(k_1, \cdots, k_v), k)$-continuous if, for two points $x$ and $x^\prime$ which are $AP_u(k_1, \cdots, k_v)$-adjacent on $\prod_1^v X_i$, $f(x)=f(x^\prime)$ or $f(x)$ is $k$-adjacent to $f(x^\prime)$.\\
	In particular, in case $k_i=k_j, i,j \in [1,v]_{\mathbb Z}$, we say that the map $f$ above is  $(AP_u(k_1,k_1, \cdots, k_1), k)$-continuous.
	
	(2) In (1) above, in case we replace the term $\lq\lq$$AP_u(k_1, \cdots, k_v)$" with $\lq$$AP_u^\ast(k_1, \cdots, k_v)$',
	we define the $(AP_u^\ast(k_1, \cdots, k_v), k)$- and $(AP_u^\ast(k_1,k_1, \cdots, k_1), k)$-continuity of the given map $f$.	
\end{definition}

The two continuities of Definition 11 can be represented as follows:
\begin{proposition} For  $u \in [1, v]_{\mathbb Z}$, the following are obtained.\\
	(1)	A function  $f: (\prod_1^v X_i, AP_u(k_1, \cdots, k_v)) \to (Y,k)$ is  $(AP_u(k_1, \cdots, k_v), k)$-continuous at a point $p \in \prod_1^v X_i, u \in [1,v]_{\mathbb Z}$ if and only if
	$$f(N_{AP_u(k_1, \cdots, k_v)}(p)) \subset N_{k}(f(p)).\eqno(6.1)$$
	A map $f:(\prod_1^v X_i, AP_u(k_1, \cdots, k_v) \to (Y, k)$ is $(AP_u(k_1, \cdots, k_v)), k)$-continuous at each point $p \in \prod_1^v X_i$, then the map $f$ is $(AP_u(k_1, \cdots, k_v), k)$-continuous.\\
	(2) In (1) above, after replacing the term $\lq\lq$$N_{AP_u(k_1, \cdots, k_v)}(p)$" by $\lq$$N_{AP_u^\ast(k_1, \cdots, k_v)}(p)$',
	we obtain the $(AP_u^\ast(k_1, \cdots, k_v), k)$- and $(AP_u^\ast(k_1,k_1, \cdots, k_1), k)$-continuity of the given map $f$.	
\end{proposition}

Using the $AP_1(k,k)$-adjacency of $X^2$ derived from $(X, k)$ and the $(AP_1(k,k), k)$-continuity, as a special case of  Definition 11, we obtain the following:

\begin{corollary}
(1)	We say that the map  $f: (X_1 \times X_2, AP_1(k_1,k_2)) \to (Y,k)$ is  $(AP_1(k_1,k_2), k)$-continuous if, for two points $x$ and $x^\prime$ which are $AP_1(k_1,k_2)$-adjacent in $X_1\times X_2$, $f(x)=f(x^\prime)$ or $f(x)$ is $k$-adjacent to $f(x^\prime)$.\\
	In particular, in case $(X_1,k_1)=(X_2,k_2)$, we say that the map $f$ above is  $(AP_1(k_1,k_1), k)$-continuous.\\
	(2) In (1) above, after replacing the term $\lq\lq$$AP_u(k_1,k_2)$" by $\lq$$AP_u^\ast(k_1,k_2)$',
	we define the $(AP_u^\ast(k_1, k_2), k)$- and $(AP_u^\ast(k_1,k_2), k)$-continuity of the given map $f$.
	\end{corollary}

Based on Corollary 6.2, the following is defined to establish an $AP_1$-$k$-group and an $AP_1^\ast$-$k$-group on a digital image $(X,k)$.
 
\begin{definition} (1)
Assume both a digital image $(X,k)$ and a group $(X, \ast)$. 
We say that the set $(X,k,\ast)$ is an $AP_1$-$k$-group if, for $(x,y) \in X^2$,
 the multiplication	
	$$\left \{
	\aligned
	& \alpha_3: (X^2, AP_1(k,k)) \to (X, k)\,\,\text{defined by}\,\,\alpha_3(x, y)= x \ast y\,\,\\
	&\text{is}\,\,(AP_1(k,k), k)\text{-continuous}
		\endaligned
	\right\}\eqno(6.2)  $$	
	
	and the inverse map with respect to the operation $\ast$
	$$\beta: (X,k) \to (X,k)\,\,\text{defined by}\,\,\beta(x)= x^{-1}\,\,\text{is}\,\,k\text{-continuous}, \,\,\,\,\,\,\,\,\,\,\,\eqno(6.3)$$
	where $x^{-1}$ indicates the inverse element of $x$ under the group $(X, \ast)$.\\
	(2) In (1) above, after replacing the term $\lq\lq$$AP_1(k,k)$" by  $\lq$$AP_1^\ast(k,k)$' we define an $AP_1^\ast$-$k$-group on $(X,k)$ and the group $(X,\ast)$.
\end{definition}

In view of Definition 12, it is clear that an $AP_1^\ast$-$k$-group is a special case of an $AP_1$-$k$-group.
In Definition 12, for an $AP_1$-$k$-group $(X, k, \ast)$, in case the given group  $(X,\ast)$ is abelian, we say that this $AP_1$-$k$-group is abelian.\\

Unlike the $DT$-groups with Types 1 and 2 mentioned earlier in Section 5, using Definition 12, we have established the $AP_1$-$k$-group and $AP_1^\ast$-$k$-group.
Compared to the two types of $DT$-groups in \cite{H10,LS1} mentioned in Section 5, Definition 12 uses the $AP_1(k,k)$-adjacency (resp. $AP_1^\ast(k,k)$-adjacency) on $X^2$ and the $(AP_u(k,k), k)$-continuity (resp. $(AP_u^\ast(k,k), k)$-continuity) of the multiplication $\alpha_3$ (resp. $\alpha_3^\prime$) of (6.4).
Namely, based on both the $AP_1(k,k)$- and $AP_1^\ast(k,k)$-adjacency on $X^2$ (see Definition 8), we use the maps (see Definition 12)
$$\alpha_3:(X^2, AP_1(k,k))\to (X, k)\,\,\text{or}\,\,\alpha_3^\prime:(X^2, AP_1^\ast(k,k))\to (X, k). \eqno(6.4)$$
Then we observe that the map $\alpha_3^\prime$ of (6.4) is equivalent to the map $\alpha_1^\prime$ of (5.1).
Namely, comparing among Definitions 3, 5, and 8, in case $(X^2, C_{k^\ast})$ exists, based on Remark 5.1,
$$\text{since}\,\,C_{k^\ast}=AP_1^\ast(k,k),\,\,\text{we obtain the equality}\,\, \alpha_1^\prime=\alpha_3^\prime, \eqno(6.5)$$
where $\alpha_1^\prime$  is referred to in (5.1).
For instance, consider the product $SC_8^{2,4} \times SC_8^{2,4} \subset {\mathbb Z}^4$. Then it has the $AP_1(8,8)$-adjacency and  $AP_1^\ast(8,8)$-adjacency, where $AP_1(8,8)\in \{k(2,4)=32, k(3,4)=64\}$ and $AP_1^\ast(8,8)=k(2,4)=32=C_{32}$ (in detail,  see Remark 3.3). Thus we have  $\alpha_1^\prime= \alpha_3^\prime$.\\

 Meanwhile, even though we follow the map $\alpha_3$ of (6.4) for the multiplication for an $AP_1$-$k$-group, we prove that the product property of an $AP_1$-$k$-group does not hold in general (see Theorem 6.7). \\
Meanwhile, unlike the map $\alpha_3$ of (6.4), if we suppose the map $\alpha_4$ as a multiplication
$$\alpha_4:(X^2, AP_2(k,k))\to (X, k), \eqno(6.6)$$
in a way similar to (6.4),
then we prove that $\alpha_4$ cannot be a multiplication for establishing an $AP_2$-$k$-group (see Remark 6.11)). 
Thus we finally confirm some utilities of the $G_{k^\ast}$-, $C_{k^\ast}$-adjacency, and $AP_1(k,k)$-adjacency 
for a product $X^2$ for establishing a $DT$-group  and an $AP_1$-$k$-group derived from $(X, k)$ (see Proposition 5.2 and Definition 12).

\begin{theorem} $({\mathbb Z}^n, k, +)$ is an abelian $AP_1$-$k$-group, where the $k$-adjacency is anyone of (2.1).
\end{theorem}
\begin{proof} (1) First of all, $({\mathbb Z}^n, +)$ is clearly a group \cite{Fr1}.
	Second, it is clear that the multiplication 
	$$\left \{
	\aligned
	& \alpha_3: ({\mathbb Z}^n \times {\mathbb Z}^n, AP_1(k,k)) \to ({\mathbb Z}^n, k)\,\,\text{defined by}\,\,\alpha_3(x, y)= x + y\,\,\\
	&\text{is}\,\,(AP_1(k,k), k)\text{-continuous}.
	\endaligned
	\right\}\eqno(6.7)  $$	
	To be specific, for each point $p \in {\mathbb Z}^{2n}$,  take a point $q \in N_{AP_1(k,k)}(p)$. Then it is clear that $\alpha_3(p)$ is $k$-adjacent to $\alpha_3(q)$.\\
				Next, it is obvious that, for each point $p \in  {\mathbb Z}^{n}$,
				the inverse map $\beta:({\mathbb Z}^n,k) \to ({\mathbb Z}^n,k)$ is $k$-continuous, where $\beta(p)=-p \in {\mathbb Z}^n$.
			\end{proof}

\begin{example}(1) $({\mathbb Z}^2, 4, +)$ is an $AP_1$-$4$-group that is equal to the $AP_1^\ast$-$4$-group.\\
	(2) $({\mathbb Z}^2, 8, +)$ is an $AP_1$-$8$-group that is equal to the $AP_1^\ast$-$8$-group.\\
\end{example}

\begin{remark}
	Not every $SC_k^{n,l}$ always has an $AP_1$-$k$-group structure. To be specific, see Proposition 4.9.
	For instance, since the product  $MSC_{18}\times MSC_{18}$ does not have any $AP_u(18,18)$-adjacency, $u \in \{1,2\}$, we cannot establish a multiplication $\alpha_3: MSC_{18}\times MSC_{18} \to MSC_{18}$ for an $(AP_u(18,18), 18)$-continuity of the multiplication.
\end{remark}

\begin{corollary} In case $SC_k^{n,l}\times SC_k^{n,l}$ has an $AP_1(k,k)$-adjacency, $SC_k^{n,l}$ can be an $AP_1$-$k$-group and an $AP_1^\ast$-$k$-group. 
\end{corollary}
\begin{proof} To establish both an  $AP_1$-$k$-group and an $AP_1^\ast$-$k$-group, the existence of an $AP_1(k,k)$-adjacency is strongly required on $X^2$. To be specific, by Remark 2.1,  for $SC_k^{n,l}$, even though $(SC_k^{n,l}, \ast)$ is a group, the product need not have an $AP_1(k,k)$-adjacency (see Proposition 4.9), which completes the proof.
	\end{proof}

	\begin{theorem}
An $AP_1^\ast$-$k$-group does not have the product property.
	Namely, assume an $AP_1^\ast$-$k_1$-group $(X, k_1, \ast_1), X \subset {\mathbb Z}^{n_1}$ and an $AP_1^\ast$-$k_2$-group $(Y, k_2, \ast_2), Y\subset {\mathbb Z}^{n_2}$. 
	Then the product $X\times Y$ need not be an $AP_1^\ast$-$k$-group $(X \times Y, k, \ast_3)$.
\end{theorem}
\begin{proof} As an example, consider the $AP_1^\ast$-$4$-group $(SC_4^{2,4}, \ast_1)$ and  the $AP_1^\ast$-$8$-group $(SC_8^{2,4}, \ast_2)$ (see Corollary 6.6). Then consider the group 
	$(SC_4^{2,4}\times SC_8^{2,4}, \ast_3) \subset {\mathbb Z}^4$ induced by the given two groups in a natural way \cite{Fr1}. Then it is clear that the product $(SC_4^{2,4} \times SC_8^{2,4}) \times (SC_4^{2,4} \times SC_8^{2,4})\subset {\mathbb Z}^8$ does not have $AP_1^\ast(4,8,4,8)$-adjacency so that we cannot have a $DT$-group structure on $SC_4^{2,4} \times SC_8^{2,4}$ associated with Definition 12.	
	\end{proof}

In view of Theorem 6.7, for two $AP_1$-$k_i$-groups, $i \in \{1,2\}$, the product $X_1 \times X_2$ need not have an $AP_1$-$k$-group for some $k$-adjacency for $X_1 \times X_2$.
\begin{corollary} An $AP_1$-$k$-group does not have the product property.
	\end{corollary}

\begin{remark} $({\mathbb Z}^2, 4, +)$ is not an $AP_2$-$4$-group.\\
\end{remark}
\begin{proof} Let us suppose the multiplication $\alpha_4: ({\mathbb Z}^2 \times {\mathbb Z}^2, AP_2(4,4)) 
	\to ({\mathbb Z}^2, 4)$. Then consider the two points $(0,0,0,0)$ and $(1,0,1,0)$ which are $AP_2(4,4)$-adjacent to each other. Then the images of $\alpha_4$ of them are $\alpha_4((0,0,0,0))=(0,0)$ and 
	 $\alpha_4((1,0,1,0))=(2,0)$ which are not $4$-adjacent. 
 \end{proof}

\begin{theorem}
(1)	 An $AP_1$-$k$-group is a generalization of a Han's $DT$-$k$-group structure based on both the $C_{k^\ast}$-adjacency on the product $X^2$ and the $(C_{k^\ast}, k)$-continuity for the multiplication $\alpha_1^\prime:(X^2, C_{k^\ast}) \to (X,k)$. \\
(2) An $AP_1^\ast$-$k$-group is equivalent to a Han's $DT$-$k$-group structure based on both the $C_{k^\ast}$-adjacency on the product $X^2$ and the $(C_{k^\ast}, k)$-continuity for the multiplication $\alpha_1^\prime:(X^2, C_{k^\ast}) \to (X,k)$.
\end{theorem}
\begin{proof} (1) Since a Han's $DT$-$k$-group structure based on both the $C_{k^\ast}$-adjacency on the product $X^2$ and the $(C_{k^\ast}, k)$-continuity for the multiplication $\alpha_1^\prime:(X^2, C_{k^\ast}) \to (X,k)$ is exactly equivalent to the 
	$AP_1^\ast$-$k$-group because  $C_{k^\ast}=AP_1^\ast(k,k)$, which completes the  proof.  \\
	(2) By Remark 5.1 and Definition 12(2), the proof is completed.	
	\end{proof}

\begin{remark} For a digital image $(X,k)$, it need not be an $AP_2$-$k$-group.
	\end{remark}
\begin{proof} For instance, $({\mathbb Z}, 2)$ is not an $AP_2$-$2$-group.
	To be specific, take the two points $p:=(0,0)$ and $q:=(1,1)$ in ${\mathbb Z}^2$.
	While they are $AP_2(2,2)$-adjacency, the images of the multiplication $\alpha_4(p)=0$ and $\alpha_4(q)=2$ are not $2$-adjacent in $({\mathbb Z}, 2)$.
	\end{proof}

	\section{\bf Summary}\label{s7}
	
	After using the $AP_u(k_1, \cdots, k_v)$- and  $AP_u^\ast(k_1, \cdots, k_v)$-adjacency,
	the paper developed the $AP_u$-$k$- and $AP_u^\ast$-$k$-group.
	Besides, after comparing among two types of $DT$-groups with Types 1 and 2 referred to in Section 5,
	we investigated some distinctions among them. 
	Finally, we confirmed that an $AP_1^\ast$-$k$-group is equivalent to a Han's $DT$-$k$-group based on both the $C_{k^\ast}$-adjacency on the product $X^2$ and the $(C_{k^\ast}, k)$-continuity for the multiplication $\alpha_1^\prime:(X^2, C_{k^\ast}) \to (X,k)$. Besides, we obverved that the $DT$-$k$-group \cite{H10} based on both the $G_{k^\ast}$-adjacency on the product $X^2$ and the $(G_{k^\ast}, k)$-continuity for the multiplication for a $DT$-$k$-group is the most generalized $DT$-$k$-group in the literature of a $DT$-group structure on a digital image $(X,k)$.

\bigskip
	{\bf Funding}: 
	National Research Foundation of~Korea funded by the Ministry of Education, Science and Technology (2019R1I1A3A03059103)
	
	\bigskip
	
	{\bf Conflicts of Interest}: The author declares no conflict of interest.
	
	\bigskip
	
	{\bf Data availability}: The author confirms that the data supporting the findings of this study are available within this article. 
	\\


\newpage
	\begin{thebibliography}{99}
	
						\bibitem{Be1} {C. Berge}, {Graphs and hypergraphs}, 2nd edition, North-Holl., Amsterdam, 1976.
			
					
								\bibitem{B1} {L. Boxer}, {Generalized normal product adjacency in digital topology}, {\em Appl.Gen.Topol.} {\bf 18}(2) (2017) 401-427.
					
				\bibitem{B2} {L. Boxer}, {Limiting sets in digital topology}, {\em Note di Matematica} {\bf 43}(2) (2023) 27-47.
			
				
	 		\bibitem{Fr1} {J.B. Fraleigh}, {A first course in abstract algebra, $7$ edition}, published by Pearson, 2002.
		
							\bibitem{H1} {S.-E. Han}, {Digital coverings and their applications},
			{\em  J. Appl. Math. Comput.} {\bf 18}(1) (2005) 487--495.
			
			\bibitem{H2} {S.-E. Han}, {Non-product property of the digital fundamental group }, {\em Inform. Sci.} {\bf  171} (2005) 73--91.
						
			\bibitem{H3}{S.-E. Han}, {Discrete Homotopy of a Closed $k$-Surface}, {\em LNCS 4040};
			Springer: Berlin, Germany, 2006;~pp.~214--225.
				
				\bibitem{H4} {S.-E. Han}, {Comparisons among digital fundamental groups and its applications}, {\em Inform. Sci.} {\bf  178} (2008) 2091--2104.
						
					\bibitem{H5} {S.-E. Han}, {Cartesian product of the universal covering property },
			{\em Acta Appl. Math.}, {\bf 108} (2009) 363–-383.
						
				\bibitem{H6} {S.-E. Han}, {An equivalent property of a normal adjacency of a digital product},
			{\em Honam Math. J.}, {\bf 36} (2014) 519--530.
			
							
			\bibitem{H7} {S.-E. Han}, {Compatible Adjacency Relations for Digital Products},
			{\em Filomat}, {\bf 31}(9) (2017) 2787--2803.
			
			\bibitem{H8} {S.-E. Han}, {Estimation of the complexity of a digital image form the viewpoint of fixed point theory} {\em Applied Mathematics and Computation} {\bf 347} (2019) 236-248.
						
			\bibitem{H9} {S.-E. Han}, {The most refined axiom for a digital covering space and its utilities}
			{\em Mathematics} {\bf 8}(1868) (2020), doi:10.3390/math811868.
			
			\bibitem{H10} {S.-E. Han}, {Digitally topological groups}, {\em Electronic Research Archive}, {\bf 30}(6) (2022) 2356-2384.
			
			\bibitem{H11} {S.-E. Han}, {$DT$-$k$-group structures on digital objects and an answer to an open problem},  {\em Discrete Mathematics, Algorithms and Applications} {\bf 16}(6) (2024) 2350074 (26 pages).	 
					 
			\bibitem{H12} {S.-E. Han}, {Development of the $AP_u(k_1,k_2, \cdots, k_v)$-adjacency for a product of multiple digital images}, {\em Submitted} (2024) (29 pages).	
			
				\bibitem{H13} {S.-E. Han}, {Remarks on the digital-topological $k$-group structures and the development of the $AP_1$-$k$- and $AP_1^\ast$-$k$-group}, {\em arXiv:6769320} (2025) (28 pages).
		
			\bibitem{H14} {F. Harary}, {Graph theory}, Addison-Wesley Publishing, Reading, MA, 1969
			
			\bibitem{KKM1} {E. Khalimsky, R. Kopperman, P.R. Meyer},
			{Computer graphics and connected topologies on finite ordered sets}, {\em Topol. Appl.}, {\bf 36(1)} (1990) 1--17.
			
			\bibitem{Ko1} {H. Kok},
			{Connected orderable spaces}, {\em Mathematisch Centrum}, ISBN 978-90-6196-088-1, Amsterdam, 1973.
			
			
			\bibitem{KR1} {T. Y. Kong, A. Rosenfeld}, {Topological Algorithms for the Digital Image Processing}, Elsevier Science, Amsterdam, 1996.
			
				\bibitem{LS1} {D.-W. Lee and P.-C. Staecker}, {Digital topological groups}, {\em Topology and its Applications} {\bf 338} (2023) 108644.
				
			\bibitem{MW1} {D. Marcus and F. Wyse et al.}, {Solution to problem 5712}, {\em Amer. Math. Monthly} {\bf 77} (1970) 1119.	
			
						\bibitem{R1} {A. Rosenfeld}, {Digital topology}, {\em Amer. Math. Monthly} {\bf 86} (1979), 76-87.
								
			
			\bibitem{R2} {A. Rosenfeld}, {Continuous functions on digital pictures},
			{\em Pattern Recognition Letters} {\bf 4} (1986) 177-184.
			
			
			\bibitem{KHL1} {J,-M. Kang, S.-E.Han, S.Lee}, {Digital products with $PN_k$-adjacencies and the almost fixed point property in $DTC_k^\blacktriangle$}  {\em AIMS(Mathematics)}
			\textbf{6}(10) (2021) 11550-11567.
			
			
			\bibitem{R1} {A. Rosenfeld}, {Connectivity in digital pictures}, {\em J. Assoc. Comput. Math.} {\bf 17} (1970) 146--160.
									
								\bibitem{R2} {A. Rosenfeld}, {Digital topology}, {\em Amer. Math. Monthly } {\bf 26} (1979) 621--630.
			
			\bibitem{R3} {A. Rosenfeld}, {Continuous functions on digital pictures},
			{\em Pattern Recognition Letters} {\bf 4} (1986) 177--184.
			
			
		\end{thebibliography}
\end{document}